\documentclass{amsart}
\usepackage{amsmath}
  \usepackage{paralist}
  \usepackage{amssymb}
 \usepackage{amsthm}
\usepackage[colorlinks=true]{hyperref}
\hypersetup{urlcolor=blue, citecolor=red}

\newtheorem{theorem}{Theorem}[section]

\newtheorem{lemma}[theorem]{Lemma}

\theoremstyle{definition}

\newtheorem{remark}[theorem]{Remark}

\numberwithin{equation}{section}

\title[Normalized solutions to lower critical Choquard equation]
      {Normalized solutions to lower critical Choquard equation with a local perturbation}

\author[Xinfu Li, Jianguang Bao, Wenguang Tang]{}


\thanks{*Corresponding author. Email Addresses:  lxylxf@tjcu.edu.cn; lxytwg@tjcu.edu.cn.}

\begin{document}
\maketitle

\centerline{\scshape Xinfu Li,\ Jianguang Bao, \ Wenguang
Tang$^{*}$}
\medskip
{\footnotesize
 \centerline{School of Science, Tianjin University of Commerce, Tianjin 300134,
P. R. China}}

\bigskip

\begin{abstract}
In this paper, we study the existence and non-existence of
normalized solutions to the lower critical Choquard equation with a
local perturbation
\begin{equation*}
\begin{cases}
-\Delta u+\lambda u=\gamma
(I_{\alpha}\ast|u|^{\frac{N+\alpha}{N}})|u|^{\frac{N+\alpha}{N}-2}u+\mu
|u|^{q-2}u,\quad \text{in}\  \mathbb{R}^N, \\
\int_{\mathbb{R}^N}|u|^2dx=c^2,
\end{cases}
\end{equation*}
where  $\gamma, \mu, c>0$, $2<q\leq 2+\frac{4}{N}$, and $\lambda\in
\mathbb{R}$ is an unknown parameter that appears as a Lagrange
multiplier. The results of this paper about this equation answer
some questions proposed by Yao, Chen, R\v{a}dulescu and Sun [Siam J.
Math. Anal., 54(3) (2022), 3696-3723]. Moreover, based on the
results obtained, we  study the multiplicity of normalized solutions
to the non-autonomous Choquard equation
\begin{equation*}
\begin{cases}
-\Delta u+\lambda u=(I_\alpha\ast [h(\epsilon
x)|u|^{\frac{N+\alpha}{N}}])h(\epsilon
x)|u|^{\frac{N+\alpha}{N}-2}u+\mu|u|^{q-2}u,\
x\in \mathbb{R}^N, \\
\int_{\mathbb{R}^N}|u|^2dx=c^2,
\end{cases}
\end{equation*}
where $\epsilon>0$, $2<q<2+\frac{4}{N}$, and $h$ is a positive and
continuous function. It is proved that the numbers of normalized
solutions are at least the numbers of global maximum points of $h$
when $\epsilon$ is small enough.
\end{abstract}

\section{Introduction and main results}

\setcounter{section}{1}
\setcounter{equation}{0}

In this paper, our  first aim is to study the existence and
non-existence of normalized solutions to the lower critical
autonomous Choquard equation with a local term
\begin{equation}\label{e1.3}
\begin{cases}
-\Delta u+\lambda u=\gamma
(I_{\alpha}\ast|u|^{\frac{N+\alpha}{N}})|u|^{\frac{N+\alpha}{N}-2}u+\mu
|u|^{q-2}u,\quad \text{in}\  \mathbb{R}^N, \\
\int_{\mathbb{R}^N}|u|^2dx=c^2,
\end{cases}
\end{equation}
where $N\geq 1$, $\gamma, \mu, c>0$, $2<q\leq 2+\frac{4}{N}$,
$\lambda\in \mathbb{R}$ is an unknown parameter that appears as a
Lagrange multiplier, $\alpha\in (0,N)$, and $I_{\alpha}$ is the
Riesz potential defined for every $x\in \mathbb{R}^N \setminus
\{0\}$ by
\begin{equation*}
I_{\alpha}(x):=\frac{A_\alpha(N)}{|x|^{N-\alpha}},\
A_\alpha(N):=\frac{\Gamma(\frac{N-\alpha}{2})}{\Gamma(\frac{\alpha}{2})\pi^{N/2}2^\alpha}
\end{equation*}
with $\Gamma$ denoting the Gamma function (see \cite{Riesz1949AM},
P.19).

One motivation driving the search for normalized solutions to
(\ref{e1.3}) is the nonlinear equation
\begin{equation}\label{e1.5}
i\partial_t\psi-\Delta
\psi-\gamma(I_\alpha\ast|\psi|^{p})|\psi|^{p-2}\psi-\mu|\psi|^{q-2}\psi=0,\
(t,x)\in \mathbb{R}\times \mathbb{R}^N,
\end{equation}
where $\frac{N+\alpha}{N}\leq p\leq \frac{N+\alpha}{(N-2)^+}$ and
$2<q\leq \frac{2N}{(N-2)^+}$. The equation (\ref{e1.5}) has several
physical origins. When $N = 3$, $p = 2$, $\alpha = 2$ and $\mu=0$,
(\ref{e1.5}) was investigated by Pekar in \cite{Pekar 1954} to study
the quantum theory of a polaron at rest. In \cite{Lieb 1977},
Choquard applied it as an approximation to Hartree-Fock theory of
one component plasma. It also arises in multiple particles systems
\cite{Gross 1996} and quantum mechanics \cite{Penrose 1996}. When $p
= 2$, equation (\ref{e1.5}) reduces to the well-known Hartree
equation. The Choquard equation (\ref{e1.5}) with or without a local
perturbation has attracted much attention nowadays, see
\cite{{Bonanno-Avenia 2014},{Cazenave 2003},{Chen-Guo
2007},{Feng-Yuan 2015},{Liu-Shi 2018},{Miao-Xu 2010}} for the local
existence, global existence, blow up and more in general dynamical
properties.

Searching for standing wave solution $\psi(t,x)=e^{-i\lambda t}u(x)$
of (\ref{e1.5})  leads to
\begin{equation}\label{e1.6}
-\Delta u+\lambda u=\gamma (I_{\alpha}\ast|u|^{p})|u|^{p-2}u+\mu
|u|^{q-2}u,\quad \text{in}\  \mathbb{R}^N.
\end{equation}
When looking for solutions to (\ref{e1.6}) one choice is to fix
$\lambda>0$ and  search for solutions to (\ref{e1.6}) as critical
points of the functional
\begin{equation*}
\bar{J}(u):=\int_{\mathbb{R}^N}\left(\frac{1}{2}|\nabla
u|^2+\frac{\lambda}{2}|u|^2-\frac{\gamma}{2p}(I_\alpha\ast
|u|^{p})|u|^{p}-\frac{\mu}{q}|u|^q\right)dx,
\end{equation*}
see for example \cite{{Li-Ma 2020},{Li-Ma-Zhang 2019},{Luo
2020},{Moroz-Schaftingen 2017}} and the references therein. Another
choice is to fix the $L^2$-norm of the unknown function $u$, that
is, to consider the problem
\begin{equation}\label{e1.55}
\begin{cases}
-\Delta u+\lambda
u=\gamma(I_\alpha\ast|u|^{p})|u|^{p-2}u+\mu|u|^{q-2}u,\ \text{in}\
\mathbb{R}^{N},\\
\int_{\mathbb{R}^N}|u|^2dx=c^2
\end{cases}
\end{equation}
with fixed $c>0$ and unknown $\lambda\in \mathbb{R}$. This is
meaningful from the physical point of view, as the $L^2$-norm is a
preserved quantity of the evolution equation (\ref{e1.5}) and the
variational characterization of such solutions is often a strong
help to analyze their orbital stability, see \cite{Cazenave-Lions82}
and the references therein. In this direction, define on
$H^1(\mathbb{R}^N)$ the energy functional
\begin{equation*}
\hat{J}(u):=\frac{1}{2}\int_{\mathbb{R}^N}|\nabla
u|^2dx-\frac{\gamma}{2p}\int_{\mathbb{R}^N}(I_\alpha\ast|u|^{p})|u|^{p}dx
-\frac{\mu}{q}\int_{\mathbb{R}^N}|u|^{q}dx.
\end{equation*}
A critical point of $\hat{J}$ constrained to
\begin{equation*}
S(c):=\left\{u\in
H^1(\mathbb{R}^N):\int_{\mathbb{R}^N}|u|^2dx=c^2\right\}
\end{equation*}
gives rise to a solution to (\ref{e1.55}). Such solution is usually
called a normalized solution, which is the aim of this paper.

The normalized solutions to the Choquard equation have been studied
extensively in these years, see
\cite{{Bartsch-Liu-Liu_2020},{Cingolani-Gallo-Tanaka-CV-22},{Cingolani-Tanaka21},{Jeanjean-Le-JDE},{Li-Ye_JMP_2014},
{Luo 2019}, {Ye 2016},{Yuan-Chen-Tang_2020}} and the references
therein. Recently, Li \cite{{Li21},{Li2022}} studied the existence,
multiplicity, orbital stability and instability of the normalized
solutions to the upper critical Choquard equation (\ref{e1.6}),
i.e., $p=\frac{N+\alpha}{N-2}$. As to the lower critical Choquard
equation, the  lower critical exponent $p=\frac{N+\alpha}{N}$  seems
to be a new feature for the Choquard equation, which is related to a
new phenomenon of ``bubbling at infinity" (see \cite{Moroz-Van
Schaftingen15}). So compared with the study developed for the upper
critical Choquard equation, the lower critical problem  seems to be
more challenging. In a recent paper \cite{Yao-Chen22}, the authors
explored new methods and ideas to study the  normalized solutions to
the lower critical Choquard equation, meanwhile they proposed some
open questions. In this paper,  we settle parts of the open
questions.

A solution $u$ to the problem (\ref{e1.3}) corresponds to a critical
point of the functional
\begin{equation}\label{e1.10}
\begin{split}
E_{q}(u):&=\frac{1}{2}\int_{\mathbb{R}^N}|\nabla u|^2dx-\frac{\gamma
N}{2(N+\alpha)}\int_{\mathbb{R}^N}(I_\alpha\ast|u|^{\frac{N+\alpha}{N}})|u|^{\frac{N+\alpha}{N}}dx\\
&\qquad-\frac{\mu}{q}\int_{\mathbb{R}^N}|u|^{q}dx
\end{split}
\end{equation}
restricted to the sphere $S(c)$. It is easy to see that $E_{q}\in
C^1(H^1(\mathbb{R}^N),\mathbb{R})$ and
\begin{equation*}
\begin{split}
E_{q}'(u)\varphi&=\int_{\mathbb{R}^N}\nabla u\nabla \varphi
dx-\gamma\int_{\mathbb{R}^N}(I_\alpha\ast|u|^{\frac{N+\alpha}{N}})|u|^{\frac{N+\alpha}{N}-2}u\varphi
dx\\
&\qquad -\mu\int_{\mathbb{R}^N}|u|^{q-2}u\varphi dx,\ \text{for\
any\ } \varphi\in H^1(\mathbb{R}^N).
\end{split}
\end{equation*}
For any $2<r<\frac{2N}{(N-2)^+}$, we define
$\eta_r:=\frac{N}{2}-\frac{N}{r}$,
\begin{equation}\label{e1.1}
\bar{S}^{-1}:=\inf_{u\in
H^1(\mathbb{R}^N)}\frac{\left(\int_{\mathbb{R}^N}|u|^2dx\right)^{\frac{r(1-\eta_r)}{2}}
\left(\int_{\mathbb{R}^N}|\nabla
u|^2dx\right)^{\frac{r\eta_r}{2}}}{\int_{\mathbb{R}^N}|u|^rdx},
\end{equation}
and
\begin{equation}\label{e1.2}
S_\alpha:=\inf_{u\in
H^1(\mathbb{R}^N)}\frac{\int_{\mathbb{R}^N}|u|^2dx}
{\left(\int_{\mathbb{R}^N}(I_\alpha\ast|u|^{\frac{N+\alpha}{N}})|u|^{\frac{N+\alpha}{N}}dx\right)^{\frac{N}{N+\alpha}}}.
\end{equation}

The main results to the equation (\ref{e1.3}) are as follows.

\begin{theorem}\label{thm1.1}
Let $N\geq 1$, $\alpha\in (0,N)$, $c,\gamma,\mu>0$, and
$2<q<2+\frac{4}{N}$. Then the infimum
\begin{equation*}
\sigma(c):=\inf_{S(c)}E_q(u)<-\frac{\gamma
N}{2(N+\alpha)}S_\alpha^{-\frac{N+\alpha}{N}}c^{\frac{2(N+\alpha)}{N}}
\end{equation*}
is achieved by $u_0 \in  S(c)$ with the following properties:

(i) $u_0$ is a real-valued positive function in $\mathbb{R}^N$,
which is radially symmetric and non-increasing;

(ii) $u_0$ is a ground state of (\ref{e1.3}) with some $\lambda_c
> \frac{\gamma N}{N+\alpha}S_\alpha^{-\frac{N+\alpha}{N}}c^{\frac{2\alpha}{N}}$.
A ground state $v$ of (\ref{e1.3}) is defined as follows:
\begin{equation*}
E_q|_{S(c)}'(v)=0\ \text{and}\ E_{q}(v)=\inf\{E_{q}(w):\ w\in S(c),\
E_q|_{S(c)}'(w)=0\}.
\end{equation*}
\end{theorem}

\begin{theorem}\label{thm1.2}
Let $N\geq 1$, $\alpha\in (0,N)$, $c,\gamma,\mu>0$, and
$q=2+\frac{4}{N}$. If $\, 0<\mu c^{4/N}\leq \frac{N+2}{N\bar{S}}$,
then (\ref{e1.3}) has no solutions.
\end{theorem}

\begin{remark}\label{rmk1.1}
(1) Theorem \ref{thm1.1} was obtained in \cite{Yao-Chen22} under the
assumption that $\mu$ is larger than some positive constant. In this
paper, by using the extremal functions of (\ref{e1.2}), we improve
their result.

(2) Theorem \ref{thm1.2} was obtained in \cite{Yao-Chen22} under the
assumption that $0<\mu c^{4/N}< \frac{N+2}{N\bar{S}}$. In this
paper, by using  the extremal functions of (\ref{e1.1}) and
(\ref{e1.2}), we can improve their result. Moreover, we give a
different proof for Theorem \ref{thm1.2} and further give some
insights for other cases (see Remark \ref{rmk1.3}).
\end{remark}

Our second aim is to study the multiplicity of normalized solutions
to the non-autonomous Choquard equation
\begin{equation}\label{e1.7}
\begin{cases}
-\Delta u+\lambda u=(I_\alpha\ast [h(\epsilon
x)|u|^{\frac{N+\alpha}{N}}])h(\epsilon
x)|u|^{\frac{N+\alpha}{N}-2}u+\mu|u|^{q-2}u,\
x\in \mathbb{R}^N, \\
\int_{\mathbb{R}^N}|u|^2dx=c^2,
\end{cases}
\end{equation}
where $\epsilon>0$, $2<q<2+\frac{4}{N}$, $N, c, \mu, \lambda,
\alpha, I_\alpha$ are as in (\ref{e1.3}), and  $h$ satisfies
the following conditions:\\
($h_1$) $h\in C(\mathbb{R}^N,\mathbb{R})$ and $0<h_{inf}:=\inf_{x\in
\mathbb{R}^N}h(x)\leq \max_{x\in \mathbb{R}^N}h(x)=:h_{{max}}$;\\
($h_2$) $h_{\infty}:=\lim_{|x|\to +\infty}h(x)<h_{{max}}$;\\
($h_3$) $h^{-1}(h_{{max}})=\{a_1,a_2,\cdots,a_l\}$ with $a_1=0$ and
$a_j\neq a_i$ if $i\neq j$.

There are many researches concerning the non-normalized solutions to
the non-autonomous lower critical Choquard equation in these years,
see \cite{{Li-Li-Wu 21},{Moroz-Van Schaftingen15},{Zhou-Liu-Zhang
22}} and the references therein.  Recently, there is a lot of work
going on with normalized solutions to the non-autonomous Choquard
equation, see \cite{{Dinh 20},{Guo-Luo-Wang},{Li-Zhao-Wang
19},{Lions 1984},{Shi 21},{Xiao-Geng-Wang 20}}. However, as to our
knowledge, there are no paper considering the multiplicity of
normalized solutions to the non-autonomous lower critical Choquard
equation. In this paper, based the results obtained to (\ref{e1.3}),
and using the method of \cite{{Alves-ZAMP22},{Cao-Noussair 96}}, we
study this problem and obtain the following result to (\ref{e1.7}).

\begin{theorem}\label{thm1.3}
Let $N\geq 1$, $\alpha\in (0,N)$, $c,\epsilon,\mu>0$,
$2<q<2+\frac{4}{N}$, and $h$ satisfy ($h_1$)-($h_3$). Then there
exists $\epsilon_0>0$ such that (\ref{e1.7}) admits at least $l$
couples $(u_j,\lambda_j)\in H^1(\mathbb{R}^N)\times \mathbb{R}$ of
weak solutions for $\epsilon\in (0,\epsilon_0)$ with
$\int_{\mathbb{R}^N}|u_j|^2dx=c^2$, $\lambda_j>0$ and
$\mathcal{E}_{\epsilon}(u_j)<0$ for $j=1,2,\cdots,l$, where
\begin{equation*}
\begin{split}
\mathcal{E}_{\epsilon}(u):&=\frac{1}{2}\int_{\mathbb{R}^N}|\nabla
u|^2dx-\frac{\mu}{q}\int_{\mathbb{R}^N}|u|^{q}dx\\
&\qquad-\frac{ N}{2(N+\alpha)}\int_{\mathbb{R}^N}(I_\alpha\ast
[h(\epsilon x)|u|^{\frac{N+\alpha}{N}}])h(\epsilon
x)|u|^{\frac{N+\alpha}{N}}dx.
\end{split}
\end{equation*}
\end{theorem}

\begin{remark}\label{rmk1.2}
 The same result as Theorem \ref{thm1.3}  can be proved for the problem
\begin{equation*}
\begin{cases}
-\Delta u+\lambda u=\gamma(I_\alpha\ast
|u|^{\frac{N+\alpha}{N}})|u|^{\frac{N+\alpha}{N}-2}u+ h(\epsilon
x)|u|^{q-2}u,\
x\in \mathbb{R}^N, \\
\int_{\mathbb{R}^N}|u|^2dx=c^2.
\end{cases}
\end{equation*}
\end{remark}

\medskip

This paper is organized as follows. In Section 2, we cite some
preliminary lemmas used in the study. In Section 3, we study the
 autonomous Choquard equation (\ref{e1.3}) and  in Section 4, we
study the  non-autonomous Choquard equation (\ref{e1.7}).

\medskip

\textbf{Notation}:  For $t\geq 1$, the  $L^t$-norm of $u\in
L^t(\mathbb{R}^N)$ is denoted by $\|u\|_t$. The usual norm of $u\in
H^1(\mathbb{R}^N)$ is denoted by $\|u\|$. \ $o_n(1)$ denotes a real
sequence with $o_n(1) \to 0$ as $n\to +\infty$. `$\rightarrow$'
denotes  strong convergence and `$\rightharpoonup$' denotes weak
convergence. $B_{r}(x_0):=\{x\in\mathbb{R}^N:|x-x_0|<r\}$.
$C,C_1,C_2,\cdots$ denote any positive constant, whose value is not
relevant and may be change from line to line.

\section{Preliminaries }
\setcounter{section}{2} \setcounter{equation}{0}

The following Gagliardo-Nirenberg inequality can be found in
(\cite{Weinstein 1983}, Theorem B).

\begin{lemma}\label{lem3.1}
Let $N\geq 1$, $2<r<\frac{2N}{(N-2)^+}$, and $\bar{S}$ be defined in
(\ref{e1.1}). Then $\bar{S}>0$ and  the minimizer of $\bar{S}$ is
$$u(x)=\left(\frac{\nu_0}{\mu_0}\right)^{\frac{1}{r-2}}w(\sqrt{\nu_0}(x-y)),\  y\in \mathbb{R}^N,$$
where $\nu_0=\frac{4}{N(r-2)}(1-\frac{(r-2)(N-2)}{4})$,
$\mu_0=\frac{4}{N(r-2)}$, and $w$ is the unique positive radial
solution of the equation
\begin{equation*}
-\Delta u+u=|u|^{r-2}u,\ \text{in}\ \mathbb{R}^N.
\end{equation*}
\end{lemma}

The following Hardy-Littlewood-Sobolev inequalities can be found in
(\cite{Lieb-Loss 2001}, Theorem 4.3).

\begin{lemma}\label{lem HLS}
Assume that $N\geq 1$, $\alpha\in (0,N)$, $p, r>1$ with
$1/p+(N-\alpha)/N+1/r=2$. Let $u\in L^p(\mathbb{R}^N)$ and $v\in
L^r(\mathbb{R}^N)$. Then there exists a sharp constant
$C(N,\alpha,p)$, independent of $u$ and $v$, such that
\begin{equation*}
\left|\int_{\mathbb{R}^N}\int_{\mathbb{R}^N}\frac{u(x)v(y)}{|x-y|^{N-\alpha}}dxdy\right|\leq
C(N,\alpha,p)\|u\|_p\|v\|_r.
\end{equation*}
\end{lemma}

\begin{lemma}\label{lem3.2}
Let $N\geq 1$, $\alpha\in (0,N)$, and $S_\alpha$ be defined in
(\ref{e1.2}). Then $S_\alpha>0$ and  the minimizer of $S_\alpha$ is
$$u(x)=a\left(\frac{\delta}{\delta^2+|x-y|^2}\right)^{N/2},\ a\in
\mathbb{R},\ \delta>0,\ y\in \mathbb{R}^N.$$
\end{lemma}

The following Br\'{e}zis-Lieb-type lemma for nonlocal nonlinearities
can be found in  (\cite{Moroz-Schaftingen JFA 2013}, Lemma 2.4).

\begin{lemma}\label{lem3.4}
Let $N\geq 1$, $\alpha\in (0,N)$ and $\{u_n\}$ be a bounded sequence
in $L^2(\mathbb{R}^N)$. If $u_n \to u$ a.e. in $\mathbb{R}^N$, then
\begin{equation*}
\begin{split}
\int_{\mathbb{R}^N}(I_\alpha\ast|u_n|^{\frac{N+\alpha}{N}})|u_n|^{\frac{N+\alpha}{N}}dx
&=\int_{\mathbb{R}^N}(I_\alpha\ast|u_n-u|^{\frac{N+\alpha}{N}})|u_n-u|^{\frac{N+\alpha}{N}}dx\\
&\quad
+\int_{\mathbb{R}^N}(I_\alpha\ast|u|^{\frac{N+\alpha}{N}})|u|^{\frac{N+\alpha}{N}}dx+o_n(1).
\end{split}
\end{equation*}
\end{lemma}

\section{The autonomous problem}
\setcounter{section}{3} \setcounter{equation}{0}

\subsection{ The case $2<q<2+\frac{4}{N}$}

\begin{lemma}\label{lem2.1}
Let $N\geq 1$, $\alpha\in (0,N)$, $c,\gamma,\mu>0$, and
$2<q<2+\frac{4}{N}$. Then

(1) $E_q$ defined in (\ref{e1.10}) is bounded below and coercive on
$S(c)$;

(2) $\sigma (c)<-\frac{\gamma
N}{2(N+\alpha)}S_\alpha^{-\frac{N+\alpha}{N}}c^{\frac{2(N+\alpha)}{N}}$,
where $\sigma (c)$ is defined in  Theorem \ref{thm1.1};

(3) For $0<c_1<c_2$, there holds   $\frac{c_1^2}{c_2^2}\sigma
(c_2)<\sigma (c_1)$.
\end{lemma}

\begin{proof}
(1) By (\ref{e1.1}) and (\ref{e1.2}), for any $u\in S(c)$, one has
\begin{equation*}
E_q(u)\geq \frac{1}{2}\|\nabla u\|_2^2-\frac{\gamma
N}{2(N+\alpha)}S_\alpha^{-\frac{N+\alpha}{N}}c^{\frac{2(N+\alpha)}{N}}-\frac{\mu}{q}\bar{S}c^{q(1-\eta_q)}\|\nabla
u\|_2^{q\eta_q},
\end{equation*}
which implies that $E_q$ is bounded below and coercive on $S(c)$ for
$q\eta_q<2$.

(2) By Lemma \ref{lem3.2}, we choose $v$ such that
\begin{equation*}
S_\alpha\left(\int_{\mathbb{R}^N}(I_\alpha\ast|v|^{\frac{N+\alpha}{N}})|v|^{\frac{N+\alpha}{N}}dx\right)^{\frac{N}{N+\alpha}}=
\int_{\mathbb{R}^N}|v|^2dx.
\end{equation*}
Define $u:=\frac{cv}{\|v\|_2}$ and $u_{\tau}(x):={\tau}^{N/2}u(\tau
x)$ for $\tau>0$, then $u_\tau\in S(c)$ for all $\tau>0$. By direct
calculations, we deduce that
\begin{equation*}
\begin{split}
E_{q}(u_\tau)&=\frac{1}{2}\tau^2\|\nabla u\|_2^2-\frac{\gamma
N}{2(N+\alpha)}\int_{\mathbb{R}^N}(I_\alpha\ast|u|^{\frac{N+\alpha}{N}})|u|^{\frac{N+\alpha}{N}}dx
-\frac{\mu}{q}\tau^{q\eta_q}\|u\|_q^q\\
&=\frac{1}{2}\tau^2\|\nabla
u\|_2^2-\frac{\mu}{q}\tau^{q\eta_q}\|u\|_q^q-\frac{\gamma
N}{2(N+\alpha)}S_\alpha^{-\frac{N+\alpha}{N}}c^{\frac{2(N+\alpha)}{N}}\\
&<-\frac{\gamma
N}{2(N+\alpha)}S_\alpha^{-\frac{N+\alpha}{N}}c^{\frac{2(N+\alpha)}{N}}
\end{split}
\end{equation*}
for $\tau>0$ small enough. This proves (2).

(3) Let $\{ u_n\}  \subset  S(c_1)$ be a bounded minimizing sequence
for $\sigma (c_1)$ and let $\nu:=\frac{c_2}{c_1}$. Then $\nu>1$,
 $\{\nu u_n\}\subset S(c_2)$ and
\begin{equation*}
\begin{split}
E_q(\nu u_n)&=\nu^2 E_{q}(u_n)+\frac{\gamma
N}{2(N+\alpha)}(\nu^2-\nu^{\frac{2(N+\alpha)}{N}})\int_{\mathbb{R}^N}(I_\alpha\ast|u_n|^{\frac{N+\alpha}{N}})|u_n|^{\frac{N+\alpha}{N}}dx\\
&\qquad +\frac{\mu}{q}(\nu^2-\nu^q)\|u_n\|_q^q,
\end{split}
\end{equation*}
which implies  that $E_q(\nu u_n) < \nu^2 E_{q}(u_n)$. Consequently,
$\sigma (\nu c_1) \leq \nu^2\sigma (c_1)$, where the equality holds
if and only if
\begin{equation*}
\int_{\mathbb{R}^N}(I_\alpha\ast|u_n|^{\frac{N+\alpha}{N}})|u_n|^{\frac{N+\alpha}{N}}dx+\|u_n\|_q^q\to
0\ \text{as}\ n\to +\infty.
\end{equation*}
But this is not possible, since otherwise we find that
\begin{equation*}
0>\sigma(c_1)=\lim_{n\to +\infty} E_q(u_n)= \lim_{n\to +\infty}
\left(\frac{1}{2}\|\nabla u_n\|_2^2+o_n(1)\right)\geq 0.
\end{equation*}
Hence $\sigma (\nu c_1) < \nu^2\sigma (c_1)$, that is,
$\frac{c_1^2}{c_2^2}\sigma (c_2)<\sigma (c_1)$.  The proof is
complete.
\end{proof}

\begin{lemma}\label{lem2.2}
Assume that $N\geq 1$, $\alpha\in (0,N)$, $c,\gamma,\mu>0$, and
$2<q<2+\frac{4}{N}$. Let $\{ u_n\} \subset S(c)$ be a sequence such
that  $E_q(u_n)\rightarrow \sigma(c)$. Then the sequence $\{ u_n\}$
is relatively compact in $H^1(\mathbb{R}^N)$ up to translations.
\end{lemma}

\begin{proof}
It follows  from Lemma \ref{lem2.1}(1) that  $\{u_n\}$ is bounded in
$H^1(\mathbb{R}^N)$. We claim that there exits $\beta>0$ such that
\begin{equation}\label{e2.1}
\lim_{n\rightarrow +\infty}\sup_{y\in \mathbb{R}^N} \int_{B_R(y)} |
u_n|^2dx \geq \beta\ \text{for\ some}\  R
>0.
\end{equation}
Suppose the contrary. Then by the Lions lemma (\cite{Willem 1996},
Lemma 1.21), $u_n \rightarrow 0$ in $L^s(\mathbb{R}^N)$ for $2 < s <
\frac{2N}{(N-2)^+}$. Using this together with (\ref{e1.1}) and
(\ref{e1.2}) gives
\begin{equation*}
\begin{split}
\sigma (c) + o_n(1)&= E_{q}(u_n)\\
&=\frac{1}{2}\|\nabla u_n\|_2^2-\frac{\gamma
N}{2(N+\alpha)}\int_{\mathbb{R}^N}(I_\alpha\ast|u_n|^{\frac{N+\alpha}{N}})|u_n|^{\frac{N+\alpha}{N}}dx
+o_n(1)\\
&\geq \frac{1}{2}\|\nabla u_n\|_2^2-\frac{\gamma
N}{2(N+\alpha)}S_\alpha^{-\frac{N+\alpha}{N}}c^{\frac{2(N+\alpha)}{N}}+o_n(1),
\end{split}
\end{equation*}
which contradicts Lemma \ref{lem2.1}(2). So (\ref{e2.1}) holds.

By (\ref{e2.1}), there exists $\{y_n\}\subset\mathbb{R}^N$ such that
$u_n(x+y_n)\rightharpoonup u\not\equiv 0$ in $H^1(\mathbb{R}^N)$.
Set $v_n=u_n(x+y_n)-u$. If $\|u\|_2=b\neq c$, then $b\in (0,c)$.
Setting $d_n=\|v_n\|_2$, and by using
\begin{equation*}
\|u_n(x+y_n)\|_2^2=\|v_n\|_2^2+\|u\|_2^2+o_n(1),
\end{equation*}
we obtain that  $d_n\in (0,c)$ for $n$ large enough and
$\|v_n\|_2\to d$ with $c^2=d^2+b^2$. It follows from Lemma
\ref{lem3.4}, Lemma \ref{lem2.1}(3) and (\cite{Moroz-Schaftingen JFA
2013}, Lemma 2.5) that
\begin{equation*}
\begin{split}
\sigma(c)+o_n(1)=E_q(u_n(x+y_n))&= E_q(v_n)+E_q(u)+o_n(1) \\
&\geq \sigma(d_n)+\sigma(b)+o_n(1)\\
&\geq \frac{d_n^2}{c^2}\sigma(c)+\sigma(b)+o_n(1).
\end{split}
\end{equation*}
Letting $n\to+\infty$ and using again Lemma \ref{lem2.1}(3), we find
that
\begin{equation*}
\sigma(c)\geq
\frac{d^2}{c^2}\sigma(c)+\sigma(b)>\frac{d^2}{c^2}\sigma(c)+\frac{b^2}{c^2}\sigma(c)=\sigma(c),
\end{equation*}
which is a contradiction. So $\|u\|_2=c$ and $u_n(x+y_n)\to u$ in
$L^2(\mathbb{R}^N)$ and then $u_n(x+y_n)\to u$ in
$L^s(\mathbb{R}^N)$ for $2\leq s<\frac{2N}{(N-2)^+}$. Consequently,
we obtain that
\begin{equation*}
\sigma(c)=\lim_{n\to +\infty}E_{q}(u_n)=\lim_{n\to
+\infty}E_{q}(u_n(x+y_n))\geq  E_{q}(u)\geq \sigma(c).
\end{equation*}
This shows that $E_{q}(u)=\sigma(c)$ and $\|\nabla u_n(x+y_n)\|_2\to
\|\nabla u\|_2$ as $n\to+\infty$, that is, $u\in S(c)$ is a
minimizer of $\sigma(c)$ and $u_n(x+y_n)\to u$ in
$H^1(\mathbb{R}^N)$. The proof is complete.
\end{proof}

\textbf{Proof of Theorem \ref{thm1.1}}. In view of  Lemma
\ref{lem2.1}(1), let $\{u_n\}\subset S(c)$ be a sequence such that
$E_{q}(u_n)\to \sigma(c)$. Then by Lemma \ref{lem2.2}, there exist a
sequence of points $\{y_n\} \subset  \mathbb{R}^N$ and a function $u
\in  S(c)$ such that up to a subsequence $u_n(\cdot+y_n)\rightarrow
u$ in $H^1(\mathbb{R}^N)$. Thus $E_q(u)=\sigma(c)$.

Let $u_0$  denote the Schwartz rearrangement of $|u|$. By using the
Riesz rearrangement inequality (see \cite{Lieb-Loss 2001})
\begin{equation*}
\begin{split}
&\|\nabla u_0\|_2\leq \|\nabla |u|\|_2\leq \|\nabla u\|_2,\
\|u_0\|_r=\|u\|_r\
\text{for\ any\ }r\geq 1,\\
&\int_{\mathbb{R}^N}(I_\alpha\ast|u_0|^{\frac{N+\alpha}{N}})|u_0|^{\frac{N+\alpha}{N}}dx\geq
\int_{\mathbb{R}^N}(I_\alpha\ast|u|^{\frac{N+\alpha}{N}})|u|^{\frac{N+\alpha}{N}}dx,
\end{split}
\end{equation*}
we deduce that $E_q(u_0)=\sigma(c)$, that is, $\sigma (c)$ is
achieved by the real-valued positive and radially symmetric
non-increasing function $u_0\in S(c)$.

It is obvious that $u_0$ is a ground state to (\ref{e1.3}) and there
exists $\lambda_c\in \mathbb{R}$ such that $E_{q}'(u_0)+\lambda_c
u_0=0$. Then
\begin{equation*}
\begin{split}
\lambda_c c^2&=-\|\nabla
u_0\|_2^2+\gamma\int_{\mathbb{R}^N}(I_\alpha\ast|u_0|^{\frac{N+\alpha}{N}})
|u_0|^{\frac{N+\alpha}{N}}dx+\mu\|u_0\|_q^q\\
&=-2\sigma(c)+\frac{\gamma
\alpha}{N+\alpha}\int_{\mathbb{R}^N}(I_\alpha\ast|u_0|^{\frac{N+\alpha}{N}})
|u_0|^{\frac{N+\alpha}{N}}dx+\frac{\mu(q-2)}{q}\|u_0\|_q^q\\
&\geq -2\sigma(c),
\end{split}
\end{equation*}
which implies that $\lambda_c > \frac{\gamma
N}{N+\alpha}S_\alpha^{-\frac{N+\alpha}{N}}c^{\frac{2\alpha}{N}}$,
where we have used Lemma \ref{lem2.1}(2). The proof is complete.

\subsection{The case $q=2+\frac{4}{N}$}\ \ \ \\

\textbf{Proof of Theorem \ref{thm1.2}}. If $u$ is a solution to
(\ref{e1.3}), by the Pohozaev identity, $u$ satisfies
\begin{equation*}
Q_q(u):=\int_{\mathbb{R}^N}|\nabla
u|^2dx-\mu\eta_q\int_{\mathbb{R}^N}|u|^{q}dx=0,
\end{equation*}
see \cite{Yao-Chen22}. So we must find solutions to (\ref{e1.3}) in
the set
\begin{equation*}
\mathcal{M}_q(c):=\{u\in S(c):Q_q(u)=0\}.
\end{equation*}

For any $u\in \mathcal{M}_q(c)$, using the Gagliardo-Nirenberg
inequality (\ref{e1.1}), we obtain that
\begin{equation}\label{e2.2}
\begin{split}
\int_{\mathbb{R}^N}|\nabla
u|^2dx=\mu\eta_q\int_{\mathbb{R}^N}|u|^{q}dx&\leq \mu\eta_q\bar{S}\|\nabla u\|_2^{q\eta_q}\|u\|_2^{q(1-\eta_q)}\\
&=\frac{\mu N}{N+2}\bar{S}c^{4/N}\|\nabla u\|_2^{2}
\end{split}
\end{equation}
for $q=2+\frac{4}{N}$, which implies that
$\mathcal{M}_q(c)=\emptyset$ if $\frac{\mu N}{N+2}\bar{S}c^{4/N}<1$.
Hence, (\ref{e1.3}) does not admit a solution if $\frac{\mu
N}{N+2}\bar{S}c^{4/N}<1$.

If $\frac{\mu N}{N+2}\bar{S}c^{4/N}=1$, then in (\ref{e2.2}), the
equality holds in the Gagliardo-Nirenberg inequality. So by Lemma
\ref{lem3.1}, $u\in S(c)$ must be some scaling transformation of the
unique positive radial solution to the equation
\begin{equation*}
-\Delta u+u=|u|^{4/N}u.
\end{equation*}
Obviously, such $u$ is  not a solution to (\ref{e1.3}). Thus,
(\ref{e1.3}) does not have a solution in the  case $\frac{\mu
N}{N+2}\bar{S}c^{4/N}=1$.

\begin{remark}\label{rmk1.3}
Let $q=2+\frac{4}{N}$. By Lemma \ref{lem3.2}, for any $u\in
\mathcal{M}_q(c)$, we have
\begin{equation*}
\begin{split}
E_q(u)&=-\frac{\gamma
N}{2(N+\alpha)}\int_{\mathbb{R}^N}(I_\alpha\ast|u|^{\frac{N+\alpha}{N}})|u|^{\frac{N+\alpha}{N}}dx\\
&\geq  -\frac{\gamma
N}{2(N+\alpha)}S_\alpha^{-\frac{N+\alpha}{N}}c^{\frac{2(N+\alpha)}{N}}
\end{split}
\end{equation*}
and  the equality holds if and only if there exist $\delta, a>0$
such that
\begin{equation*}
u=a\left(\frac{\delta}{\delta^2+|x|^2}\right)^{N/2}\ \text{and}\
u\in \mathcal{M}_q(c),
\end{equation*}
which is equivalent to
\begin{equation}\label{e2.4}
u=c\left(\int_{\mathbb{R}^N}\frac{1}{(1+|x|^2)^N}dx\right)^{-1/2}
\left(\frac{\delta}{\delta^2+|x|^2}\right)^{N/2}
\end{equation}
and
\begin{equation}\label{e2.5}
\begin{split}
\frac{\mu
N}{N+2}c^{4/N}=\frac{\left\|\frac{1}{(1+|x|^2)^{N/2}}\right\|_2^{q(1-\eta_q)}
\left\|\nabla\frac{1}{(1+|x|^2)^{N/2}}\right\|_2^{q\eta_q}}{\left\|\frac{1}{(1+|x|^2)^{N/2}}\right\|_q^q}(>\bar{S}^{-1}).
\end{split}
\end{equation}
That is, under the assumption (\ref{e2.5}), $u$ defined in
(\ref{e2.4}) belongs to $\mathcal{M}_q(c)$ and is the unique
minimizer of $E_{q}|_{\mathcal{M}_q(c)}$. However, $u$ defined  in
(\ref{e2.4}) is not a solution to (\ref{e1.3}), so we can not obtain
a solution to (\ref{e1.3}) by minimizing $E_{q}|_{\mathcal{M}_q(c)}$
if $q=2+\frac{4}{N}$ and $c, \mu$ satisfy (\ref{e2.5}).
\end{remark}

\section{The non-autonomous problem}
\setcounter{section}{4} \setcounter{equation}{0}

To show the dependence on parameters, we rewrite
\begin{equation*}
\begin{split}
J_{\gamma}(u):&=\frac{1}{2}\int_{\mathbb{R}^N}|\nabla
u|^2dx-\frac{\mu}{q}\int_{\mathbb{R}^N}|u|^{q}dx\\
&\qquad-\frac{\gamma^2
N}{2(N+\alpha)}\int_{\mathbb{R}^N}(I_\alpha\ast|u|^{\frac{N+\alpha}{N}})|u|^{\frac{N+\alpha}{N}}dx
\end{split}
\end{equation*}
and
\begin{equation*}
\Upsilon_{\gamma,c}:=\inf_{u\in S(c)}J_{\gamma}(u).
\end{equation*}
It is obvious that $\mathcal{E}_{\epsilon}(u)\geq J_{h_{max}}(u)$
for any $u\in S(c)$. By Theorem \ref{thm1.1}, the definition
\begin{equation*}
\Gamma_{\epsilon,c}:=\inf_{u\in S(c)}\mathcal{E}_{\epsilon}(u)
\end{equation*}
is well defined and  $\Gamma_{\epsilon,c}\geq \Upsilon_{h_{max},c}$.

The next two lemmas  establish the relations of
$\Gamma_{\epsilon,c}$, $\Upsilon_{h_\infty,c}$ and
$\Upsilon_{h_{max},c}$.

\begin{lemma}\label{cor3.1}
Let $N\geq 1$, $\alpha\in (0,N)$, $\mu, c>0$, $2<q<2+\frac{4}{N}$,
and $0 <\gamma_1 < \gamma_2$. Then $\Upsilon_{\gamma_2,c} <
\Upsilon_{\gamma_1,c}$. In particular,
$\Upsilon_{h_{max},c}<\Upsilon_{h_\infty,c}<0$.
\end{lemma}

\begin{proof}
Let $u\in S(c)$ satisfy $J_{\gamma_1}(u)=\Upsilon_{\gamma_1,c}$.
Then, $\Upsilon_{\gamma_2,c}\leq J_{\gamma_2}(u)<J_{\gamma_1}(u) =
\Upsilon_{\gamma_1,c}$.
\end{proof}

\begin{lemma}\label{lem4.2}
Let $N\geq 1$, $\alpha\in (0,N)$, $\epsilon,\mu, c>0$,
$2<q<2+\frac{4}{N}$,  and $h$ satisfy ($h_1$)-($h_3$). Then
$$\limsup_{\epsilon\to 0^+}\Gamma_{\epsilon,c}\leq
\Upsilon_{h_{max},c}.$$
\end{lemma}

\begin{proof}
By Theorem \ref{thm1.1}, choose  $u \in  S(c)$ such that
$J_{h_{max}}(u)=\Upsilon_{h_{max},c}$. Then
\begin{equation*}
\begin{split}
\Gamma_{\epsilon,c}\leq
\mathcal{E}_{\epsilon}(u)&=\frac{1}{2}\int_{\mathbb{R}^N}|\nabla
u|^2dx-\frac{\mu}{q}\int_{\mathbb{R}^N}|u|^{q}dx\\
&\qquad-\frac{
N}{2(N+\alpha)}\int_{\mathbb{R}^N}(I_\alpha\ast [h(\epsilon x)|u|^{\frac{N+\alpha}{N}}])h(\epsilon x)|u|^{\frac{N+\alpha}{N}}dx.\\
\end{split}
\end{equation*}
Letting $\epsilon\to 0^+$, by the Lebesgue dominated convergence
theorem, we deduce that
\begin{equation*}
\limsup_{\epsilon\to 0^+}\Gamma_{\epsilon,c}\leq
\limsup_{\epsilon\to 0^+} \mathcal{E}_{\epsilon}(u)=J_{h(0)}(u)=
J_{h_{max}}(u)=\Upsilon_{h_{max},c},
\end{equation*}
which completes the proof.
\end{proof}

By Lemmas \ref{cor3.1} and \ref{lem4.2}, there exists $\epsilon_1>0$
such that $\Gamma_{\epsilon,c}<\Upsilon_{h_\infty,c}$ for all
$\epsilon\in (0,\epsilon_1)$. In the following, we always assume
that $\epsilon\in (0,\epsilon_1)$. The next two lemmas will be used
to prove the $(PS)$ condition for $\mathcal{E}_{\epsilon}$
restricted to $S(c)$ at some levels.

\begin{lemma}\label{lem4.1}
Assume that $N\geq 1$, $\alpha\in (0,N)$, $\epsilon, \mu, c>0$,
$2<q<2+\frac{4}{N}$, and $h$ satisfies ($h_1$)-($h_3$). Let
$\{u_n\}\subset S(c)$ be such that $\mathcal{E}_{\epsilon}(u_n)\to
a$ as $n\to+\infty$ with $a<\Upsilon_{h_\infty,c}$. If
$u_n\rightharpoonup u$ in $H^1(\mathbb{R}^N)$, then $u\not\equiv 0$.
\end{lemma}

\begin{proof}
Assume by contradiction that $u\equiv0$. Then
\begin{equation*}
\begin{split}
a+o_n(1)=\mathcal{E}_{\epsilon}(u_n)=&J_{h_\infty}(u_n)+\frac{
N}{2(N+\alpha)}h_\infty^2\int_{\mathbb{R}^N}(I_\alpha\ast
|u_n|^{\frac{N+\alpha}{N}})|u_n|^{\frac{N+\alpha}{N}}dx\\
&-\frac{ N}{2(N+\alpha)}\int_{\mathbb{R}^N}(I_\alpha\ast [h(\epsilon
x)|u_n|^{\frac{N+\alpha}{N}}])h(\epsilon
x)|u_n|^{\frac{N+\alpha}{N}}dx.
\end{split}
\end{equation*}
By ($h_2$), for any given $\delta>0$, there exists $R>0$ such that
$|h(x)-h_{\infty}|\leq \delta$ for all $|x|\geq R$. This together
with Lemma \ref{lem HLS} gives that
\begin{equation*}
\begin{split}
&h_\infty^2\int_{\mathbb{R}^N}(I_\alpha\ast
|u_n|^{\frac{N+\alpha}{N}})|u_n|^{\frac{N+\alpha}{N}}dx
-\int_{\mathbb{R}^N}(I_\alpha\ast [h(\epsilon
x)|u_n|^{\frac{N+\alpha}{N}}])h(\epsilon
x)|u_n|^{\frac{N+\alpha}{N}}dx\\
&=\int_{\mathbb{R}^N}(I_\alpha\ast [(h_\infty-h(\epsilon
x))|u_n|^{\frac{N+\alpha}{N}}])h_\infty|u_n|^{\frac{N+\alpha}{N}}dx\\
&\qquad+\int_{\mathbb{R}^N}(I_\alpha\ast [h(\epsilon
x)|u_n|^{\frac{N+\alpha}{N}}])(h_\infty-h(\epsilon
x))|u_n|^{\frac{N+\alpha}{N}}dx\\
&=\int_{\mathbb{R}^N}(I_\alpha\ast
[h_\infty|u_n|^{\frac{N+\alpha}{N}}])(h_\infty-h(\epsilon
x))|u_n|^{\frac{N+\alpha}{N}}dx\\
&\qquad+\int_{\mathbb{R}^N}(I_\alpha\ast [h(\epsilon
x)|u_n|^{\frac{N+\alpha}{N}}])(h_\infty-h(\epsilon
x))|u_n|^{\frac{N+\alpha}{N}}dx\\
&\leq 2
C\|h_{max}|u_n|^{\frac{N+\alpha}{N}}\|_{\frac{2N}{N+\alpha}}\|(h_\infty-h(\epsilon
x))|u_n|^{\frac{N+\alpha}{N}}\|_{\frac{2N}{N+\alpha}}\\
\end{split}
\end{equation*}
and
\begin{equation*}
\begin{split}
&\|(h_\infty-h(\epsilon
x))|u_n|^{\frac{N+\alpha}{N}}\|_{\frac{2N}{N+\alpha}}\\
=&\left(\int_{\mathbb{R}^N}|h_\infty-h(\epsilon
x)|^{\frac{2N}{N+\alpha}}|u_n|^2dx\right)^{\frac{N+\alpha}{2N}}\\
=&\left(\int_{B_{R/\epsilon}(0)}|h_\infty-h(\epsilon
x)|^{\frac{2N}{N+\alpha}}|u_n|^2dx+\int_{B^c_{R/\epsilon}(0)}|h_\infty-h(\epsilon
x)|^{\frac{2N}{N+\alpha}}|u_n|^2dx\right)^{\frac{N+\alpha}{2N}}\\
\leq & \left(\int_{B_{R/\epsilon}(0)}|h_\infty-h(\epsilon
x)|^{\frac{2N}{N+\alpha}}|u_n|^2dx+\delta^{\frac{2N}{N+\alpha}}\int_{B^c_{R/\epsilon}(0)}|u_n|^2dx\right)^{\frac{N+\alpha}{2N}}.
\end{split}
\end{equation*}
Recalling that $\{u_n\}$ is bounded in $H^1(\mathbb{R}^N)$ and
$u_n\to  0$ in $L^t(B_{R/\epsilon}(0))$ for all $t\in[1,
\frac{2N}{(N-2)^+})$, we have
\begin{equation*}
a+o_n(1)=\mathcal{E}_{\epsilon}(u_n)\geq J_{h_\infty}(u_n)-\delta
C+o_n(1)
\end{equation*}
for some $C > 0$. Since $\delta > 0$ is arbitrary, we deduce that
$a\geq \Upsilon_{h_\infty,c}$, which is a contradiction. Thus,
$u\not\equiv 0$.
\end{proof}

\begin{lemma}\label{lem4.5}
Assume that $N\geq 1$, $\alpha\in (0,N)$, $\mu, c>0$,
$\epsilon\in(0,\epsilon_1)$, $2<q<2+\frac{4}{N}$, and $h$ satisfies
($h_1$)-($h_3$). Let $\{u_{n}\}$ be a $(PS)_a$ sequence of
$\mathcal{E}_{\epsilon}$ restricted to $S(c)$ with $a <
\Upsilon_{h_\infty,c}$ and let $u_n\rightharpoonup u_\epsilon$ in
$H^1(\mathbb{R}^N)$. If $u_n\not\rightarrow u_{\epsilon}$ in
$H^1(\mathbb{R}^N)$, then there exists $\beta>0$ independent of
$\epsilon\in (0,\epsilon_1)$ such that
\begin{equation*}
\limsup_{n\to +\infty}\|u_n-u_{\epsilon}\|_2\geq \beta.
\end{equation*}
\end{lemma}

\begin{proof}
Setting the functional $\Psi: H^1(\mathbb{R}^N)\to \mathbb{R}$ given
by
\begin{equation*}
\Psi(u)=\frac{1}{2}\int_{\mathbb{R}^N}|u|^2dx,
\end{equation*}
it follows that $S(c) =\Psi^{-1}(c^2/2)$. Then, by Willem
(\cite{Willem 1996}, Proposition 5.12), there exists
$\{\lambda_n\}\subset \mathbb{R}$ such that
\begin{equation}\label{e4.8}
\|\mathcal{E}'_{\epsilon}(u_n)+\lambda_n\Psi'(u_n)\|_{H^{-1}(\mathbb{R}^N)}\to
0\ \text{as}\ n\to+\infty.
\end{equation}

By the boundedness of $\{u_n\}$ in $H^1(\mathbb{R}^N)$, we know
$\{\lambda_n\}$ is bounded and thus, up to a  subsequence, there
exists $\lambda_{\epsilon}$ such that $\lambda_n\to
\lambda_{\epsilon}$ as $n\to+\infty$. This together with
(\ref{e4.8}) leads to
\begin{equation}\label{e4.10}
\mathcal{E}'_{\epsilon}(u_\epsilon)+\lambda_\epsilon\Psi'(u_\epsilon)=0\text{\
in \ }H^{-1}(\mathbb{R}^N)
\end{equation}
and then
\begin{equation}\label{e4.9}
\|\mathcal{E}'_{\epsilon}(v_n)+\lambda_\epsilon\Psi'(v_n)\|_{H^{-1}(\mathbb{R}^N)}\to
0\ \text{as}\ n\to+\infty,
\end{equation}
where $v_n:=u_n-u_\epsilon$. By direct calculations, we get that
\begin{equation*}
\begin{split}
\Upsilon_{h_\infty,c}&>\lim_{n\to+\infty}\mathcal{E}_{\epsilon}(u_n)\\
&=\lim_{n\to+\infty}\left(\mathcal{E}_{\epsilon}(u_n)-\frac{1}{2}\mathcal{E}_{\epsilon}'(u_n)u_n-\frac{1}{2}\lambda_n\|u_n\|_2^2+o_n(1)\right)\\
&=\lim_{n\to+\infty}\left[ \left(\frac{1}{2}-\frac{
N}{2(N+\alpha)}\right)\int_{\mathbb{R}^N}(I_\alpha\ast [h(\epsilon
x)|u_n|^{\frac{N+\alpha}{N}}])h(\epsilon
x)|u_n|^{\frac{N+\alpha}{N}}dx  \right.\\
&\qquad\qquad\qquad\left.+\left(\frac{1}{2}-\frac{1}{q}\right)\mu\int_{\mathbb{R}^N}|u_n|^qdx-\frac{1}{2}\lambda_nc^2+o_n(1)\right]\\
&\geq -\frac{1}{2}\lambda_{\epsilon}c^2,
\end{split}
\end{equation*}
which implies that
\begin{equation}\label{e4.11}
\lambda_{\epsilon}\geq -\frac{2\Upsilon_{h_\infty,c}}{c^2}>0\
\text{for\ all\ }\epsilon\in (0,\epsilon_1).
\end{equation}
By (\ref{e4.9}), we know
\begin{equation}\label{e4.15}
\begin{split}
&\int_{\mathbb{R}^N}|\nabla
v_n|^2dx+\lambda_{\epsilon}\int_{\mathbb{R}^N}|v_n|^2dx-\mu\int_{\mathbb{R}^N}|v_n|^{q}dx\\
& \qquad -\int_{\mathbb{R}^N}(I_\alpha\ast [h(\epsilon
x)|v_n|^{\frac{N+\alpha}{N}}])h(\epsilon
x)|v_n|^{\frac{N+\alpha}{N}}dx=o_n(1),
\end{split}
\end{equation}
which combined with (\ref{e4.11}) gives that
\begin{equation}\label{e4.3}
\begin{split}
\int_{\mathbb{R}^N}|\nabla
v_n|^2dx&-\frac{2\Upsilon_{h_\infty,c}}{c^2}\int_{\mathbb{R}^N}|v_n|^2dx\\
&\leq h_{max}^2\int_{\mathbb{R}^N}(I_\alpha\ast
|v_n|^{\frac{N+\alpha}{N}})|v_n|^{\frac{N+\alpha}{N}}dx+\mu\int_{\mathbb{R}^N}|v_n|^{q}dx+o_n(1).
\end{split}
\end{equation}

If $u_n\not\rightarrow u_{\epsilon}$ in $H^1(\mathbb{R}^N)$, that
is, $v_n\not\rightarrow 0$ in $H^1(\mathbb{R}^N)$, by (\ref{e1.1}),
(\ref{e1.2}) and (\ref{e4.3}), we deduce that
\begin{equation*}
\begin{split}
\int_{\mathbb{R}^N}|\nabla
v_n|^2dx&-\frac{2\Upsilon_{h_\infty,c}}{c^2}\int_{\mathbb{R}^N}|v_n|^2dx\\
&\leq
h_{max}^2S_\alpha^{-\frac{N+\alpha}{N}}\|v_n\|^{\frac{2(N+\alpha)}{N}}+\mu
\bar{S}\|v_n\|^q+o_n(1).
\end{split}
\end{equation*}
So there exists $C>0$ independent of $\epsilon$ such that
$\|v_n\|\geq C$ and then by (\ref{e4.3})
\begin{equation}\label{e4.4}
\limsup_{n\to
+\infty}\left(h_{max}^2\int_{\mathbb{R}^N}(I_\alpha\ast
|v_n|^{\frac{N+\alpha}{N}})|v_n|^{\frac{N+\alpha}{N}}dx+\mu\int_{\mathbb{R}^N}|v_n|^{q}dx\right)\geq
C.
\end{equation}
Noting that
\begin{equation}\label{e4.2}
\begin{split}
\Upsilon_{h_\infty,c}>\lim_{n\to+\infty}\mathcal{E}_{\epsilon}(u_n)&\geq
 \frac{1}{2}\|\nabla u_n\|_2^2-\frac{h_{max}^2
N}{2(N+\alpha)}S_\alpha^{-\frac{N+\alpha}{N}}c^{\frac{2(N+\alpha)}{N}}\\
&\qquad-\frac{\mu}{q}\bar{S}c^{q(1-\eta_q)}\|\nabla
u_n\|_2^{q\eta_q},
\end{split}
\end{equation}
we obtain that $\{u_n\}$ is uniformly bounded in $H^1(\mathbb{R}^N)$
for any $\epsilon\in(0,\epsilon_1)$. This together with
(\ref{e1.1}), (\ref{e1.2}) and (\ref{e4.4}) gives that there exists
$\beta
> 0$ independent of $\epsilon\in (0,\epsilon_1)$ such that
\begin{equation}\label{e4.5}
\limsup_{n\to +\infty}\|v_n\|_2\geq \beta.
\end{equation}
The proof is complete.
\end{proof}

Now we give the compactness lemma.

\begin{lemma}\label{lem4.3}
Let $N\geq 1$, $\alpha\in (0,N)$, $\mu, c>0$,
$\epsilon\in(0,\epsilon_1)$, $2<q<2+\frac{4}{N}$, $h$ satisfy
($h_1$)-($h_3$), $\beta$ be obtained in Lemma \ref{lem4.5},
$$\rho_0:=\min\left\{\Upsilon_{h_\infty,c}
-\Upsilon_{h_{max},c},\frac{\beta^2}{c^2}(\Upsilon_{h_\infty,c}-\Upsilon_{h_{max},c})\right\}.$$
Then $\mathcal{E}_{\epsilon}$ satisfies the $(PS)_a$ condition
restricted to $S(c)$ if $a <\Upsilon_{h_{max},c}+\rho_0$.
\end{lemma}

\begin{proof}
Let $\{u_{n}\}\subset S(c)$ be a $(PS)_a$ sequence of
$\mathcal{E}_{\epsilon}$ restricted to $S(c)$. It follows from
$a<\Upsilon_{h_\infty,c}$ and (\ref{e4.2}) that  $\{u_n\}$ is
bounded in $H^1(\mathbb{R}^N)$. Let $u_n\rightharpoonup u_\epsilon$
in $H^1(\mathbb{R}^N)$. By Lemma \ref{lem4.1}, $u_\epsilon\not\equiv
0$. Set $v_n:=u_n-u_{\epsilon}$. If $u_n\to u_\epsilon$ in
$H^1(\mathbb{R}^N)$, the proof is complete. If $u_n\not\rightarrow
u_{\epsilon}$ in $H^1(\mathbb{R}^N)$ for some $\epsilon\in
(0,\epsilon_1)$, by Lemma \ref{lem4.5},
\begin{equation*}
\limsup_{n\to +\infty}\|v_n\|_2\geq \beta.
\end{equation*}
Set $b=\|u_{\epsilon}\|_2,\ d_n = \|v_n\|_2$ and suppose that
$\|v_n\|_2\to d$, then  we get $d\geq \beta>0$ and  $c^2 = b^2 +
d^2$. From $d_n\in(0, c)$ for $n$ large enough, we have
\begin{equation}\label{e4.6}
a+o_n(1)=\mathcal{E}_{\epsilon}(u_n)=
\mathcal{E}_{\epsilon}(v_n)+\mathcal{E}_{\epsilon}(u_{\epsilon})+o_n(1).
\end{equation}
Since $v_n\rightharpoonup 0$ in $H^1(\mathbb{R}^N)$, similarly to
the proof of Lemma \ref{lem4.1}, we deduce that
\begin{equation}\label{e4.7}
\mathcal{E}_{\epsilon}(v_n)\geq J_{h_\infty}(v_n)-\delta C+o_n(1)
\end{equation}
for any $\delta>0$, where $C>0$ is a constant independent of
$\delta$ and $n$. By (\ref{e4.6}), (\ref{e4.7}) and Lemma
\ref{lem2.1}(3), we obtain that
\begin{equation*}
\begin{split}
a+o_n(1)=\mathcal{E}_{\epsilon}(u_n)&\geq
J_{h_\infty}(v_n)+\mathcal{E}_{\epsilon}(u_{\epsilon})-\delta C+o_n(1)\\
&\geq \Upsilon_{h_\infty,d_n}+\Upsilon_{h_{max},b}-\delta C+o_n(1)\\
&\geq
\frac{d_n^2}{c^2}\Upsilon_{h_\infty,c}+\frac{b^2}{c^2}\Upsilon_{h_{max},c}-\delta
C+o_n(1).
\end{split}
\end{equation*}
Letting $n\to +\infty$, by the arbitrariness of $\delta>0$, we
obtain that
\begin{equation*}
\begin{split}
a&\geq
\frac{d^2}{c^2}\Upsilon_{h_\infty,c}+\frac{b^2}{c^2}\Upsilon_{h_{max},c}\\
&=\Upsilon_{h_{max},c}+\frac{d^2}{c^2}(\Upsilon_{h_\infty,c}-\Upsilon_{h_{max},c})\\
&\geq
\Upsilon_{h_{max},c}+\frac{\beta^2}{c^2}(\Upsilon_{h_\infty,c}-\Upsilon_{h_{max},c}),
\end{split}
\end{equation*}
which contradicts $a<\Upsilon_{h_{max},c}+\rho_0$. Thus, we must
have $u_n \to u_{\epsilon}$ in $H^1(\mathbb{R}^N)$.
\end{proof}

In what follows, we fix $\tilde{\rho}, \tilde{r} > 0$ satisfying:

\smallskip

$\bullet$\ $\overline{B_{\tilde{\rho}}(a_i)}\cap
\overline{B_{\tilde{\rho}}(a_j)}=\emptyset$ for $i\neq j$ and
$i,j\in\{1,\cdots,l\}$;

\smallskip

$\bullet$\ $\cup_{i=1}^{l}B_{\tilde{\rho}}(a_i) \subset
B_{\tilde{r}}(0)$;

\smallskip

$\bullet$\
$K_{\frac{\tilde{\rho}}{2}}=\cup_{i=1}^{l}\overline{B_{\frac{\tilde{\rho}}{2}}(a_i)}$.

We also set the function $Q_\epsilon:
H^1(\mathbb{R}^N)\backslash\{0\}\to \mathbb{R}^N$ by
\begin{equation*}
Q_\epsilon(u):=\frac{\int_{\mathbb{R}^N}\chi(\epsilon
x)|u|^2dx}{\int_{\mathbb{R}^N}|u|^2dx},
\end{equation*}
where $\chi:\mathbb{R}^N\to \mathbb{R}^N$ is given by
\begin{equation*}
\chi(x):=\left\{
\begin{array}{ll}
x, &\ \text{if}\ |x|\leq \tilde{r},\\
\tilde{r}\frac{x}{|x|}, &\ \text{if}\ |x|>\tilde{r}.
\end{array}
\right.
\end{equation*}

The next two lemmas will be useful to get important ($PS$) sequences
for $\mathcal{E}_{\epsilon}$ restricted to $S(c)$.

\begin{lemma}\label{lem4.6}
Let $N\geq 1$, $\alpha\in (0,N)$, $\epsilon, \mu, c>0$,
$2<q<2+\frac{4}{N}$, $h$ satisfy ($h_1$)-($h_3$), and $\rho_0$ be
defined in Lemma \ref{lem4.3}. Then there exist
$\epsilon_2\in(0,\epsilon_1], \rho_1\in(0,\rho_0]$ such that if
$\epsilon\in(0, \epsilon_2)$, $u\in S(c)$ and
$\mathcal{E}_{\epsilon}(u)\leq \Upsilon_{h_{max},c}+\rho_1$, then
\begin{equation*}
Q_\epsilon(u)\in  K_{\frac{\tilde{\rho}}{2}}.
\end{equation*}
\end{lemma}

\begin{proof}
If the lemma does not occur, there must be $\rho_n\to 0,
\epsilon_n\to 0$ and $\{u_n\}\subset S(c)$ such that
\begin{equation}\label{e4.16}
\mathcal{E}_{\epsilon_n}(u_n)\leq \Upsilon_{h_{max},c}+\rho_n\
\text{and}\ Q_{\epsilon_n}(u_n)\not\in K_{\frac{\tilde{\rho}}{2}}.
\end{equation}
Consequently,
\begin{equation*}
\Upsilon_{h_{max},c}\leq  J_{h_{max}}(u_n)\leq
\mathcal{E}_{\epsilon_n}(u_n)\leq \Upsilon_{h_{max},c}+\rho_n,
\end{equation*}
then
\begin{equation*}
\{u_n\}\subset S(c)\ \text{and}\ J_{h_{max}}(u_n)\to
\Upsilon_{h_{max},c}.
\end{equation*}
According to Lemma \ref{lem2.2},  there exists $\{y_n\}\subset
\mathbb{R}^N$ such that $v_n(x) := u_n(x+ y_n)\to v$ in
$H^1(\mathbb{R}^N)$ for some $v\in S(c)$. Now we will study the two
cases: (I) $|\epsilon_ny_n|\to +\infty$ and (II) $\epsilon_ny_n\to
y$ for some $y\in \mathbb{R}^N$.

If (I) holds, the limit $v_n \to v$ in $H^1(\mathbb{R}^N)$ provides
\begin{equation*}
\begin{split}
\mathcal{E}_{\epsilon_n}(u_n)&=\frac{1}{2}\int_{\mathbb{R}^N}|\nabla
v_n|^2dx-\frac{\mu}{q}\int_{\mathbb{R}^N}|v_n|^{q}dx\\
&\quad -\frac{ N}{2(N+\alpha)}\int_{\mathbb{R}^N}(I_\alpha\ast
[h(\epsilon_n
x+\epsilon_ny_n)|v_n|^{\frac{N+\alpha}{N}}])h(\epsilon_n
x+\epsilon_ny_n)|v_n|^{\frac{N+\alpha}{N}}dx\\
&\to J_{h_\infty}(v)\ \text{as}\ n\to+\infty.
\end{split}
\end{equation*}
Since $\mathcal{E}_{\epsilon_n}(u_n)\leq
\Upsilon_{h_{max},c}+\rho_n$, we deduce that
\begin{equation*}
\Upsilon_{h_\infty,c}\leq J_{h_\infty}(v)\leq \Upsilon_{h_{max},c},
\end{equation*}
which contradicts $\Upsilon_{h_\infty,c}>\Upsilon_{h_{max},c}$ in
Lemma \ref{cor3.1}.

Now if (II) holds, then
\begin{equation*}
\begin{split}
\mathcal{E}_{\epsilon_n}(u_n)\to J_{h(y)}(v)\ \text{as}\
n\to+\infty,
\end{split}
\end{equation*}
which combined with $\mathcal{E}_{\epsilon_n}(u_n)\leq
\Upsilon_{h_{max},c}+\rho_n$ gives  that
\begin{equation*}
\Upsilon_{h(y),c}\leq J_{h(y)}(v)\leq \Upsilon_{h_{max},c}.
\end{equation*}
By Lemma \ref{cor3.1}, we must have $h(y)=h_{max}$ and $y=a_i$ for
some $i=1,2,\cdots, l$. Hence,
\begin{equation*}
\begin{split}
Q_{\epsilon_n}(u_n)=\frac{\int_{\mathbb{R}^N}\chi(\epsilon_n
x)|u_n|^2dx}{\int_{\mathbb{R}^N}|u_n|^2dx}
&=\frac{\int_{\mathbb{R}^N}\chi(\epsilon_n
x+\epsilon_ny_n)|v_n|^2dx}{\int_{\mathbb{R}^N}|v_n|^2dx}\\
&\to
\frac{\int_{\mathbb{R}^N}\chi(y)|v|^2dx}{\int_{\mathbb{R}^N}|v|^2dx}=a_i\in
K_{\frac{\tilde{\rho}}{2}},
\end{split}
\end{equation*}
which implies that $Q_{\epsilon_n}(u_n)\in
K_{\frac{\tilde{\rho}}{2}}$ for $n$ large enough. That contradicts
(\ref{e4.16}). The proof is complete.
\end{proof}

From now on, we will use the following notations:

\smallskip

$\bullet$\ $\theta_{\epsilon}^i:=\{u\in S(c):|Q_\epsilon(u)-a_i|\leq
\tilde{\rho}\}$;

\smallskip

$\bullet$\ $\partial\theta_{\epsilon}^i:=\{u\in
S(c):|Q_\epsilon(u)-a_i|=\tilde{\rho}\}$;

\smallskip

 $\bullet$\
$\beta_{\epsilon}^{i}:=\inf_{u\in\theta_{\epsilon}^i}\mathcal{E}_{\epsilon}(u)$;

\smallskip

$\bullet$\
$\tilde{\beta}_{\epsilon}^{i}:=\inf_{u\in\partial\theta_{\epsilon}^i}\mathcal{E}_{\epsilon}(u)$.

\begin{lemma}\label{lem4.7}
Let $N\geq 1$, $\alpha\in (0,N)$, $\epsilon, \mu, c>0$,
$2<q<2+\frac{4}{N}$, $h$ satisfy ($h_1$)-($h_3$), $\epsilon_2$ and
$\rho_1$ be obtained in Lemma \ref{lem4.6}. Then there exists
$\epsilon_3\in (0,\epsilon_2]$ such that
\begin{equation*}
\beta_{\epsilon}^i<\Upsilon_{h_{max},c}+\frac{\rho_1}{2}\
\text{and}\ \beta_{\epsilon}^{i}<\tilde{\beta}_{\epsilon}^{i},\
\text{for\ any\ }\epsilon\in(0,\epsilon_3).
\end{equation*}
\end{lemma}

\begin{proof}
Let $u\in S(c)$ be such that
\begin{equation*}
J_{h_{max}}(u)=\Upsilon_{h_{max},c}.
\end{equation*}
For $1\leq i\leq l$, we define
\begin{equation*}
\hat{u}_{\epsilon}^{i}(x):=u\left(x-\frac{a_i}{\epsilon}\right),\
x\in \mathbb{R}^N.
\end{equation*}
Then $\hat{u}_{\epsilon}^{i}\in S(c)$ for all $\epsilon>0$ and
$1\leq i\leq l$. Direct calculations give that
\begin{equation*}
\begin{split}
\mathcal{E}_{\epsilon}(\hat{u}_{\epsilon}^{i})&=\frac{1}{2}\int_{\mathbb{R}^N}|\nabla
u|^2dx-\frac{\mu}{q}\int_{\mathbb{R}^N}|u|^{q}dx\\
&\quad -\frac{ N}{2(N+\alpha)}\int_{\mathbb{R}^N}(I_\alpha\ast
[h(\epsilon x+a_i)|u|^{\frac{N+\alpha}{N}}])h(\epsilon
x+a_i)|u|^{\frac{N+\alpha}{N}}dx,
\end{split}
\end{equation*}
and then
\begin{equation}\label{e4.1}
\lim_{\epsilon\to
0^+}\mathcal{E}_{\epsilon}(\hat{u}_{\epsilon}^{i})=J_{h(a_i)}(u)=J_{h_{max}}(u)
=\Upsilon_{h_{max},c}.
\end{equation}
Note that
\begin{equation*}
Q_{\epsilon}(\hat{u}_{\epsilon}^i)=\frac{\int_{\mathbb{R}^N}\chi(\epsilon
x+a_i)|u|^2dx}{\int_{\mathbb{R}^N}|u|^2dx}\to a_i\text{\ as}\
\epsilon\to 0^+.
\end{equation*}
So $\hat{u}_{\epsilon}^i\in \theta_{\epsilon}^i$ for $\epsilon$
small enough, which combined with (\ref{e4.1}) implies that there
exists $\epsilon_3\in(0,\epsilon_2]$ such that
\begin{equation*}
\beta_{\epsilon}^i<\Upsilon_{h_{max},c}+\frac{\rho_1}{2}\text{\ for\
any}\ \epsilon\in (0,\epsilon_3).
\end{equation*}

For any $v\in \partial\theta_{\epsilon}^i$, that is, $v\in S(c)$ and
$|Q_\epsilon(v)-a_i|=\tilde{\rho}$, we obtain that
$|Q_\epsilon(v)\not\in K_{\frac{\tilde{\rho}}{2}}$. Thus, by Lemma
\ref{lem4.6},
\begin{equation*}
\mathcal{E}_{\epsilon}(v)>\Upsilon_{h_{max},c}+\rho_1\ \text{for\
all}\ v\in
\partial\theta_{\epsilon}^i\ \text{and}\ \epsilon\in (0,\epsilon_3),
\end{equation*}
which implies that
\begin{equation*}
\tilde{\beta}_{\epsilon}^{i}=\inf_{v\in\partial\theta_{\epsilon}^i}
\mathcal{E}_{\epsilon}(v)\geq \Upsilon_{h_{max},c}+\rho_1\
\text{for\ all}\ \epsilon\in(0,\epsilon_3).
\end{equation*}
Thus,
\begin{equation*}
\beta_{\epsilon}^{i}<\tilde{\beta}_{\epsilon}^{i}\ \text{for\ all\ }
\epsilon\in (0,\epsilon_3).
\end{equation*}
\end{proof}

Now we are ready to prove Theorem \ref{thm1.3}.

\smallskip

\textbf{Proof of Theorem \ref{thm1.3}}. Set
$\epsilon_0:=\epsilon_3$, where $\epsilon_3$ is obtained in Lemma
\ref{lem4.7}. Let $\epsilon\in (0,\epsilon_0)$. By Lemma
\ref{lem4.7}, for each $i\in \{1,2,\cdots, l\}$, we can use the
Ekeland variational principle to find a sequence $\{u_n^{i}\}\subset
\theta_{\epsilon}^i$ satisfying
\begin{equation*}
\mathcal{E}_{\epsilon}(u_n^i)\to \beta_{\epsilon}^i\ \text{and}\
\|\mathcal{E}_{\epsilon}|_{S(c)}'(u_n^i)\|\to 0\ \text{as}\ n\to
+\infty,
\end{equation*}
that is, $\{u_n^i\}_n$ is a $(PS)_{\beta_{\epsilon}^i}$ sequence for
$\mathcal{E}_{\epsilon}$ restricted to $S(c)$. Since
$\beta_{\epsilon}^i< \Upsilon_{h_{max},c}+\rho_0$, it follows from
Lemma \ref{lem4.3} that there exists  $u^i$ such that $u_n^i\to u^i$
in $H^1(\mathbb{R}^N)$. Thus
\begin{equation*}
u^i\in \theta_{\epsilon}^i,\
\mathcal{E}_{\epsilon}(u^i)=\beta_{\epsilon}^i<0\ \text{and}\
\mathcal{E}_{\epsilon}|_{S(c)}'(u^i)=0.
\end{equation*}

As
\begin{equation*}
Q_{\epsilon}(u^i)\in \overline{B_{\tilde{\rho}}(a_i)},\
Q_{\epsilon}(u^j)\in \overline{B_{\tilde{\rho}}(a_j)},
\end{equation*}
and
\begin{equation*}
\overline{B_{\tilde{\rho}}(a_i)}\cap
\overline{B_{\tilde{\rho}}(a_j)}=\emptyset\ \text{for}\ i\neq j,
\end{equation*}
we conclude that $u^i\not\equiv u^j$ for $1\leq i, j \leq  l$ and
$i\neq j$. Therefore, $\mathcal{E}_{\epsilon}$ has at least $l$
nontrivial critical points for all $\epsilon\in  (0, \epsilon_0)$.
So there exists $\lambda_i\in \mathbb{R}$ such that
\begin{equation*}
-\Delta u^i+\lambda_i u^i=(I_\alpha\ast [h(\epsilon
x)|u^i|^{\frac{N+\alpha}{N}}])h(\epsilon
x)|u^i|^{\frac{N+\alpha}{N}-2}u^i+\mu|u^i|^{q-2}u^i,\quad x\in
\mathbb{R}^N.
\end{equation*}

By using $\mathcal{E}_{\epsilon}(u^i)=\beta_{\epsilon}^i<0$ and
$\mathcal{E}_{\epsilon}'(u^i)u^i+\lambda_ic^2=0$, we obtain that
\begin{equation*}
\begin{split}
\frac{1}{2}\lambda_ic^2&=-\mathcal{E}_{\epsilon}(u^i)+\left(\frac{1}{2}-\frac{1}{q}\right)\mu
\int_{\mathbb{R}^N}|u^i|^qdx\\
&\qquad+\left(\frac{1}{2}-\frac{
N}{2(N+\alpha)}\right)\int_{\mathbb{R}^N}(I_\alpha\ast [h(\epsilon
x)|u^i|^{\frac{N+\alpha}{N}}])h(\epsilon
x)|u^i|^{\frac{N+\alpha}{N}}dx,
\end{split}
\end{equation*}
which implies that $\lambda_i>0$ for $i=1,2,\cdots,l$. The proof is
complete.

\bigskip

\textbf{Acknowledgements.}  This work is supported by the National
Natural Science Foundation of China (No. 12001403).



\end{document}